\documentclass[twoside,11pt]{amsart}

\usepackage{amsmath,amssymb,amsbsy,amstext,amsxtra,latexsym}
\usepackage{graphicx}
\usepackage{hyperref}

\newtheorem{lemma}{Lemma}[section]
\newtheorem{proposition}{Proposition}[section]

\newtheorem{example}{Example}[section]
\newtheorem{remark}{Remark}[section]

\def\mP{{\mathbb P}}
\def\mQ{{\mathbb Q}}
\def\mC{{\mathbb C}}

\def\mZ{{\mathbb Z}}
\def\cO{{\mathcal O}}

\def\n{\noindent}

\def\m{\medskip}

\def\log{\operatorname{log}}

\def\cK{{\mathcal K}}

\def\log{\operatorname{log}}

\begin{document}

\title[Construction of surfaces of general type]
{Construction of surfaces of general type from elliptic surfaces
via $\mathbb{Q}$-Gorenstein smoothing }

\author{JongHae Keum}

\address{Korea Institute for Advanced Study, 207-43 Cheongnyangni-dong,
         Seoul 130-722, Korea}

\email{jhkeum@kias.re.kr}

\author{Yongnam Lee}

\address{Department of Mathematics, Sogang University,
         Sinsu-dong, Mapo-gu, Seoul 121-742, and
         Korea Institute for Advanced Study, 207-43 Cheongnyangni-dong,
         Seoul 130-722, Korea}

\email{ynlee@sogang.ac.kr}

\author{Heesang Park}

\address{Korea Institute for Advanced Study, 207-43 Cheongnyangni-dong,
         Seoul 130-722, Korea}

\email{hspark@kias.re.kr}

\date{Nov 14, 2010}



\begin{abstract}
We present methods to construct interesting surfaces of general
type via $\mathbb{Q}$-Gorenstein smoothing of a singular surface
obtained from an elliptic surface. By applying our methods to
special Enriques surfaces, we construct new examples of a minimal
surface of general type with $p_g=0$,
$\pi_1=\mathbb{Z}/2\mathbb{Z}$, and $K^2\le 4$.
\end{abstract}

\maketitle


\section{Introduction}
\label{sec-1}

A simply connected elliptic surface $S$ with a section is called an
$E(n)$ surface if $\chi({\mathcal O}_S)=n>0$ and $c_1^2(S)=0$. An
$E(n)$ surface has topological Euler characteristic $c_2=12n$. All
$E(n)$ surfaces are diffeomorphic for fixed $n$, and an $E(n)$
surface is symplectically isomorphic to the fiber sum of $n$ copies
of a rational elliptic surface $E(1)$. Recall that an $E(1)$ surface
is obtained from $\mP^2$ by blowing up the base points of a pencil
of cubics, and an $E(2)$ surface is an elliptic K3 surface with a
section.

Recently, Jongil Park and the second named author constructed a
simply connected minimal surface of general type with $p_g=0$ and
$K^2=2$ via $\mQ$-Gorenstein smoothing of a singular rational
surface \cite{LP1}. This singular rational surface is obtained by
contracting linear chains of rational curves in a blow-up of an
$E(1)$ surface with singular fibers of special type. The other
constructions of surfaces of general type with $p_g=0$ via
$\mQ$-Gorenstein smoothing given in \cite{LP2}, \cite{PPS1},
\cite{PPS2}, \cite{PPS3}, use different $E(1)$ surfaces, but all
employ the same arguments as in \cite{LP1} to prove the vanishing
$H^2(T^0_X)=0$, which is a key ingredient to guarantee the
existence of a $\mQ$-Gorenstein smoothing.

We remark that several statements in \cite{LP1} can be generalized
to the case of $E(n)$ surfaces and to the case of elliptic
surfaces without a section.

\m

\n{\bf Question.}  Is it possible to construct an interesting
complex surface via $\mQ$-Gorenstein smoothing of a singular
surface obtained by contracting linear chains of rational curves
in a blow-up of an $E(n)$ surface with $n\ge 2$, or of an
 Enriques surface?

\m In this paper, we will treat mainly the case of Enriques
surfaces. Since an Enriques surface has multiple fibres, the
method of \cite{LP1} for the case of an $E(1)$ surface cannot be
applied directly to prove the existence of a global
$\mQ$-Gorenstein smoothing. We overcome this difficulty by passing
to the K3-cover (an $E(2)$ surface) and then by showing that the
obstruction space of the corresponding singular surface has
trivial invariant part under the covering involution.

Using some special Enriques surfaces, we are able to construct
minimal surfaces of general type with $p_g=0$, $\pi_1=\mZ/2\mZ$
and $K^2=1, 2, 3, 4$. Each of our examples has ample canonical
class, i.e., contains no $(-2)$-curve. There have been constructed
several examples of a minimal surface of general type with
$p_g=0$, $\pi_1=\mZ/2\mZ$ and $K^2=1$, by Barlow \cite{Bar}, by
Inoue \cite{Inoue} and recently by Bauer and Pignatelli \cite{BP}.
Each of these examples contains a $(-2)$-curve. We do not know if
our example with $K^2=1$ is deformation equivalent to one of these
examples. As for examples with $p_g=0$, $\pi_1=\mZ/2\mZ$ and
$K^2\ge 2$, the existence is known only for $K^2=3$. In fact,
Cartwright and Steger \cite{CS} recently found examples with
$p_g=0$, $\pi_1=\mZ/2\mZ$ and $K^2=3$ by taking the minimal
resolution of the quotient of a fake projective plane by an order
3 automorphism. By \cite{K08}, the quotient of a fake projective
plane by an order 3 automorphism has $p_g=0$, $K^2=3$, and 3
singular points of type $\dfrac{1}{3}(1,2)$. They computed the
fundamental groups of all possible quotients to find that some
quotients have $\pi_1=\mZ/2\mZ$. Our example with $K^2=3$ is
different from any of their examples, but we do not know if ours
is deformation equivalent to one of theirs. A minimal surface of
general type with $p_g=0$, $K^2=2$ and $H_1=\mZ/2\mZ$ was
constructed in \cite{LP2}, but it is not known if it actually has
$\pi_1=\mZ/2\mZ$. Table 1 of \cite{BCP} gives a list of minimal
surfaces of general type with $p_g = 0$ and $K^2\le 7$ available
in the literature.

We remark that our method cannot produce minimal surfaces of
general type with $p_g = 0$ and $K^2\ge 5$. The reason is that
singular surfaces $X$ appearing in our construction satisfy the
vanishing $H^2(T^0_X)=0$, which we need to ensure the existence of
a global $\mQ$-Gorenstein smoothing. By the upper semi-continuity,
the condition $H^2(T^0_X)=0$ implies $H^2(X_t, T_{X_t})=0$ for a
general member $X_t$ of a $\mQ$-Gorenstein smoothing. Since
\[h^1(X_t, T_{X_t})-h^2(X_t, T_{X_t})=10\chi(\cO_{X_t})-2K_{X_t}^2,\]
the dimension of the deformation space of $X_t$ is equal to
$h^1(X_t, T_{X_t})=10-2K_{X_t}^2$, and hence there is no
nontrivial deformation of $X_t$ if $K_{X_t}^2\ge 5$.

\m The case of $E(n)$ with $n\ge 4$ were worked out in \cite{LP3}
and \cite{Lee}.  The case of $E(3)$ will be treated in the last
section. A key ingredient in the case of $E(n)$ is to show that
there is a $\mQ$-Gorenstein smoothing of singular points
simultaneously even if there is an obstruction to $\mQ$-Gorenstein
smoothing for each singular point.

\m
Throughout this paper, we follow Kodaira's notation for singular
fibers of elliptic fibration \cite{Kodaira}, and we work over the
field of complex numbers.


\section{The case of elliptic K3 surfaces with a section}
\label{sec-2}

In this section, we give a sufficient condition for the existence
of a $\mQ$-Gorenstein smoothing of a singular surface obtained
from a K3 surface with a section.  This will be used in our main
construction in Section 3.

\m
Let $Y$ be a K3 surface admitting an elliptic fibration with a
section whose singular fibers are either reducible or of type
$\mathrm{I_1}$ (nodal). Assume that it has a fibre of type
$\mathrm{I_1}$, and let $F_Y$ be such a fibre. Let $\pi: Z\to Y$
be the blow-up at the node of $F_Y$. Let $F$ be the proper
transform of $F_Y$ and $E$ the exceptional curve, i.e. the total
transform of $F_Y$ is $F+2E$. Let $S_1, \ldots, S_\ell\subset Z$
be the proper transform of sections in $Y$. They are
$(-2)$-curves, not meeting $E$. Let $G_1,\ldots, G_k$ be
$(-2)$-curves in the union of singular fibers. Assume that the
support of $\cup_{i=1}^k G_i$ does not contain the support of a
whole fiber, and that the sum $S_1+\cdots +
S_\ell+G_1+\cdots+G_k+F+E$ is a normal crossing divisor.

\begin{proposition}
\label{prop 2.1} With the assumptions and the notation as above,
assume further that $S_1,\ldots, S_\ell$, $G_1, \ldots, G_k$, $F,
E$
are numerically independent in the Picard group of $Z$. Then\\
$H^2(Z, T_Z(-\log(S_1+\cdots + S_\ell+G_1+\cdots+G_k+F)))=0$.
\end{proposition}

\begin{proof} We denote by $C$ the sum $S_1+\cdots + S_\ell+G_1+\cdots+G_k$, i.e.,  $C:=
S_1+\cdots + S_\ell+G_1+\cdots+G_k.$ By the Serre duality, it is
enough to show $H^0(Z, \Omega_Z^1(\log (C+F)(K_Z)))=0$. Note
that the canonical divisor $K_Z=E$, since $K_Y=\cO_Y$. By an
abuse of notation, we abbreviate $\cO_{S_1}\oplus\cdots \oplus
\cO_{S_\ell}\oplus\cO_{G_1}\oplus\cdots\oplus\cO_{G_k}$ to
$\cO_C$.

The proof uses the following commutative diagram and the snake
lemma.
\[\begin{array}{ccccccc}
& 0 & & 0& &  & \\
& \downarrow & & \downarrow & &  & \\
0\to & \Omega_Z^1 & \to & \Omega_Z^1(E) & \to &
\Omega_Z^1(E)\otimes\cO_E &\to 0 \\
& \downarrow & & \downarrow & & \downarrow & \\
0\to & \Omega_Z^1(\log(C+F+E)) & \to & \Omega_Z^1(\log(C+F))(E) &
\to & \cK &\to 0 \\
& \downarrow & & \downarrow & & \downarrow & \\
0\to\cO_E\to & \cO_C\oplus\cO_F\oplus\cO_E & \to &
\cO_C(E)\oplus\cO_F(E) & \to & \mC_p\oplus\mC_q &\to
0\\
& \downarrow & & \downarrow & & \downarrow & \\
& 0 & & 0& & 0 &
\end{array}\]
where $\cK$ is the cokernel of the map from
$\Omega_Z^1(\log(C+F+E))$ to $\Omega_Z^1(\log(C+F))(E)$ and $p, q$
are intersection points of $F$ and $E$. We have the short exact
sequence
\[0\to \cO_E\to \Omega_Z^1(E)\otimes\cO_E \to \Omega_E^1(E)\to 0.\]
Then by the snake lemma, we have the short exact sequence
\[0\to \Omega_E^1(E)\to \cK\to \mC_p\oplus\mC_q\to 0,\]
and we get $\cK=\cO_E(-1)$. Since $H^0(Z, \Omega_Z^1)=0$ and the
first Chern class map from $H^0(\cO_C\oplus\cO_F\oplus\cO_E)$ to
$H^1(Z, \Omega_Z^1)$ is injective by the assumption of $C$, we get
the vanishing $H^0(Z, \Omega_Z^1(\log(C+F+E)))=0$. And then we have
the vanishing $H^0(Z, \Omega_Z^1(\log(C+F))(E))=0$.
\end{proof}

We can keep the vanishing of the cohomology under the process of
blowing up at a point by the following standard fact: Let $V$ be a
nonsingular surface and let $D$ be a simple normal crossing
divisor in $V$. Let $f: V'\to V$ be the blow-up of $V$ at a point
p on $D$. Let  $D'$ be the reduced divisor of the total transform
of $D$. Then $h^2(V',T_{V'}(-\log D'))=h^2(V, T_V(-\log D))$.
Therefore, we get

\begin{proposition}
\label{prop 2.2} With the same assumptions and notation as in
Proposition 2.1, we denote $D_Z:=\sum_{i=1}^\ell S_i +\sum_{i=1}^k
G_i+F$. Let $\tau': Z'\to Z$ be a successive blowing-up of points
on $D_Z$. Let $D_{Z'}$ be the reduced divisor of the total
transform of $D_Z$ or the reduced divisor of the total transform
of  $D_Z$ minus some $(-1)$-curves. Then $H^2(Z', T_{Z'}(-\log
D_{Z'}))=0.$
\end{proposition}

Note that an $E(2)$ surface can be constructed as a double cover of
an $E(1)$ surface. By using the double covering $E(2)\to E(1)$,
together with the methods developed in \cite{LP1}, one can produce
simply connected minimal surfaces of general type with $p_g=1$ and
$q=0$. For example, such surfaces with $1\le K^2\le 6$ are
constructed in \cite{PPS4}.


\section{Construction of surfaces of general type
with $p_g=q=0$ and $\pi_1=\mZ/2\mZ$ from Enriques surfaces}
\label{sec-3}

Recall that every Enriques surface admits an elliptic fibration,
and every elliptic fibration on an Enriques surface has exactly
two multiple fibers, both of which have multiplicity 2. Every
smooth rational curve on an Enriques surface is a $(-2)$-curve.

A singular fibre $F$ of an elliptic fibration is said to be of {\it
additive type}, if the group consisting of simple points of $F$
contains the additive group $\mC$, and of {\it multiplicative type},
if the group consisting of simple points of $F$ contains the
multiplicative group $\mC^*$. In the Kodaira's notation for singular
fibers of elliptic fibration, a fibre of type $\mathrm{II, III, IV,
IV^*, III^*, II^*, I_n^*} (n\geq 0)$ is of additive type, and a
fibre of type $\mathrm{I_n} (n\geq 1)$ is of multiplicative type. An
additive type fibre is always a non-multiple fibre, and a fibre of
multiplicity $m\ge 2$ must be of type $m\mathrm{I_n} (n\geq 0)$,
i.e., its reduced structure must be of type $\mathrm{I_n} (n\geq
0)$.

Let $W$ be an Enriques surface and $f:W\to \mP^1$ be an elliptic
fibration on it. A smooth rational curve on $W$ is called a {\it
$2$-section}, if it intersects a fibre of $f$ with multiplicity 2.
Let $V$ be the K3 cover of $W$, and $g:V\to \mP^1$ the elliptic
fibration induced by the elliptic fibration $f: W\to \mP^1$, i.e.,
the normalization of the fibre product of $f: W\to \mP^1$ and the
double cover $\mP^1\to \mP^1$ branched at the base points of the
two multiple fibres of $f$. A fibre of $f$ is non-multiple iff it
splits into two fibres of $g$ of the same type. A multiple fibre
of $f$ of type $2\mathrm{I_n}$ does not split and gives a fibre of
$g$ of type $\mathrm{I_{2n}}$.

\begin{lemma}\label{bi} If an elliptic fibration on an Enriques surface has a
$2$-section $S$, then for each fibre $F$, $S$ passes through $F$ in
two distinct smooth points, if $F$ is not a multiple fibre; one
smooth point of $F_{red}$, if $F=2F_{red}$ is a multiple fibre.
\end{lemma}

\begin{proof} Let $W$ be an Enriques surface and
$f: W\to \mP^1$ be an elliptic fibration on it. Let $V$ be the K3
cover of $W$. Since the double cover $V \to W$ is unramified, the
$2$-section $S$ splits into two sections $S_1, S_2$ of the
elliptic fibration $g:V\to \mP^1$ induced by the elliptic
fibration $f: W\to \mP^1$. Each $S_i$ passes through each fibre of
$g$ in a smooth point. This implies the result.
\end{proof}

\begin{lemma}\label{Enr-I9} If an elliptic fibration on an Enriques surface has a
singular fiber of  type $\mathrm{I_9}$ or $2\mathrm{I_9}$, then it
has three singular fibers of type $\mathrm{I_1}$ or $2\mathrm{I_1}$.
In particular, the fibration always has at least one singular fiber
of type $\mathrm{I_1}$.
\end{lemma}

\begin{proof} Let $W$ be an Enriques surface and $f:W\to \mP^1$ be
an elliptic fibration on it. Let $J(f):J(W)\to \mP^1$ be the
Jacobian fibration of $f$. This is an elliptic fibration with a
section having singular fibres of the same type, without
multiplicity, as those of $f: W\to \mP^1$ (\cite{CD}, Theorem
5.3.1). In particular, the surface $J(W)$ is an $E(1)$ surface.
Now assume that $f$ has a singular fiber of type $\mathrm{I_9}$ or
$2\mathrm{I_9}$. Then, $J(f)$ has a singular fiber of type
$\mathrm{I_9}$. The result follows from the following lemma.
\end{proof}

\begin{lemma}\label{E1-I9} If an $E(1)$ surface has a singular
fiber of type $\mathrm{I_9}$, then it has three singular fibers of
type $\mathrm{I_1}$.
\end{lemma}

\begin{proof}
Note that an $E(1)$ surface has Picard number 10 and topological
Euler characteristic $c_2=12$.  Since the elliptic fibration has a
singular fiber of type $\mathrm{I_9}$, all other singular fibres
must be irreducible. On the other hand, in this case the
Mordell-Weil group, i.e., the group of sections of the elliptic
fibration has order 3. The homomorphism from the Mordell-Weil group
to the group consisting of simple points of every singular fibre is
injective. A singular fibre contains a 3-torsion point iff it is of
type $\mathrm{IV^*}$ or $\mathrm{IV}$ or $\mathrm{I_n} (n\geq 1)$.
So an irreducible singular fibre containing a 3-torsion point must
be of type $\mathrm{I_1}$. Finally, a singular fiber of type
$\mathrm{I_n}$ has Euler number $n$.
\end{proof}

The following will be used in proving the ampleness of $K_X$ of a
singular surface $X$ obtained by contracting chains of smooth
rational curves on a smooth surface $Z$.

\begin{lemma}\label{KC} Let $Y$ be a smooth surface and $Z$ a blow-up of $Y$ at points $p_1,\ldots,
p_k$, possibly infinitely near. Let $C$ be an irreducible curve on
$Z$. Assume that $C$ has 1-dimensional image $C'$ in $Y$. Then
$K_Z.C=K_Y.C' +\sum m_i$, where  $m_i$ is the multiplicity of $C'$
at the point $p_i$. If in addition $Y$ is an Enriques surface,
then $K_Z.C=\sum m_i\ge 0$.
\end{lemma}

Let $\bar Y$ be an Enriques surface admitting an elliptic
fibration $f:\bar Y\to\mP^1$ with a $2$-section whose singular
fibers are either reducible or of type $m\mathrm{I_1}$ ($m=1$ or
2). Let $\bar S_1,\ldots, \bar S_\ell$ be $2$-sections. These are
$(-2)$-curves. Let
$$C_{\bar Y}:=\bar S_1+\ldots + \bar S_\ell+\bar G_1+ \ldots+ \bar
G_k,$$ where $\bar G_i$ is a $(-2)$-curve contained in a
non-multiple singular fiber. We assume that the support of
$\cup_{i=1}^k \bar G_i$ does not contain the support of a whole
singular fiber. We also assume that the elliptic fibration $f:\bar
Y\to\mP^1$ has a singular fiber $F_{\bar Y}$ of type
$\mathrm{I_1}$. By Lemma 3.1, no component of $C_{\bar Y}$ passes
through the node of $F_{\bar Y}$. Let $$\pi: \bar Z \to \bar Y$$
be the blow-up of $\bar Y$ at the node of $F_{\bar Y}$. Let $\bar
F$ be the proper transform of $F_{\bar Y}$ and $\bar E$ be the
$(-1)$-curve. Let $C_{\bar Z}$ be the proper transform of $C_{\bar
Y}$. Since the exceptional curve $\bar E$ does not meet the locus
of $C_{\bar Z}$, we use the same notation $\bar S_1,\ldots, \bar
S_\ell, \bar G_1, \ldots, \bar G_k$ for their proper transforms.
That is, $$C_{\bar Z}=\bar S_1+\cdots +\bar S_\ell+ \bar G_1+
\cdots+ \bar G_k.$$ We consider the unramified double cover $$p:
Z\to \bar Z$$ induced by the line bundle $L$ of $\bar Z$,
$$L=(\text{one multiple fiber})_{\text{red}}-(\text{the other
multiple fiber})_{\text{red}}.$$ Note that $L^2=\cO_{\bar Z}$ and
$K_{\bar Z}=\bar E+L$. Let $Y\to \bar Y$ be the K3 cover, and
$g:Y\to \mP^1$ be the elliptic fibration induced by $f$. The
surface $Z$ is also obtained by blowing up $Y$ at the nodes of the
two singular fibers of type $\mathrm{I_1}$ lying over $F_{\bar
Y}$. Let $E_1, E_2$ be the two $(-1)$-curves on $Z$. Then
$$K_Z=p^*(K_{\bar Z}+L)=p^*\bar E=E_1+E_2.$$
Let $C^1_Z, C^2_Z$ be the inverse image of $C_{\bar Z}$ in $Z$,
$F_1, F_2$ the inverse image of $\bar F$, $S_i^1, S_i^2$ the
inverse image of $\bar S_i$, and $G_i^1, G_i^2$ the inverse image
of $\bar G_i$. We note that $S_i^1, S_i^2$ are sections, and  for
$j=1, 2$ $$C^j_Z=S_1^j+\cdots +S_\ell^j+ G_1^j+ \cdots+ G_k^j.$$
Here we also use abuse of notation, $$\cO_{C_{\bar Z}}:=\cO_{\bar
S_1}\oplus\cdots\oplus\cO_{\bar S_\ell}\oplus\cO_{\bar G_1}
\oplus\cdots\oplus\cO_{\bar G_k},$$ and  for $j=1, 2$
$$\cO_{C^j_Z}:=\cO_{S_1^j}\oplus\cdots \oplus
\cO_{S_\ell^j}\oplus\cO_{G_1^j}\oplus \cdots \oplus \cO_{G_k^j}.$$
Assume that the divisor $C_{\bar Z}+\bar F+\bar E$ is a simple
normal crossing divisor. Then so is the divisor $C^j_{Z}+F_j+E_j$
for $j=1,2$. We have
\[p_*(\Omega^1_Z(K_Z))=
p_*(p^*\Omega^1_{\bar Z}(K_Z))=p_*p^*(\Omega^1_{\bar Z}(K_{\bar
Z}+L))= \Omega_{\bar Z}^1(K_{\bar Z})\oplus\Omega^1_{\bar
Z}(K_{\bar Z}+L).\] Tensoring with $K_Z$ the short exact sequence
\[ 0\to \Omega^1_Z \to \Omega^1_Z(\log(C^1_Z+C^2_Z+F_1+F_2))\to \cO_{C^1_Z}
\oplus\cO_{C^2_Z}\oplus\cO_{F_1}\oplus\cO_{F_2}\to 0,\] we get the
short exact sequence
\[ 0\to \Omega^1_Z(K_Z)\to \Omega^1_Z(\log(C^1_Z+C^2_Z+F_1+F_2))(K_Z)\to \cO_{C^1_Z}
\oplus\cO_{C^2_Z}\oplus\cO_{F_1}(E_1)\oplus\cO_{F_2}(E_2)\to 0\]
because all curves in the support of $C^1_Z\cup C^2_Z$ are smooth
rational curves that do not meet $E_1$ and $E_2$. Similarly, we
have two exact sequences of sheaves in $\bar Z$,
\[ 0\to \Omega^1_{\bar Z}(K_{\bar Z})\to \Omega^1_{\bar Z}(\log(C_{\bar Z}+\bar F))
(K_{\bar Z})\to
\cO_{C_{\bar Z}}\oplus\cO_{\bar F}(\bar E)\to 0,\]
\[ 0\to \Omega^1_{\bar Z}(K_{\bar Z}+L)\to \Omega^1_{\bar Z}(\log(C_{\bar Z}+\bar F))
(K_{\bar Z}+L)\to \cO_{C_{\bar Z}}\oplus\cO_{\bar F}(\bar E)\to
0.\] Since $p : Z\to {\bar Z}$ is an unramified double cover,
$p_*$ is an exact functor. Therefore we have
\[p_*(\Omega^1_Z(\log(C^1_Z+C^2_Z+F_1+F_2))(K_Z))= \Omega^1_{\bar
Z}(\log(C_{\bar Z}+\bar F))(K_{\bar Z})\oplus\Omega^1_{\bar
Z}(\log(C_{\bar Z}+\bar F)) (K_{\bar Z}+L).\] By a similar
argument as in the proof of Proposition 2.1, we have
\[H^0(Z, \Omega^1_Z(\log(C^1_Z+C^2_Z+F_1+F_2))(E_1+E_2))=H^0(Z,
\Omega^1_Z(\log(C^1_Z+C^2_Z+F_1+F_2+E_1+E_2))).\] The involution
$\iota$ induced from the double cover $p: Z\to\bar Z$ acts on
$H^0(Z, \Omega^1_Z(\log(C^1_Z+C^2_Z+F_1+F_2+E_1+E_2)))$ and the
$\iota$-invariant subspace is isomorphic to $H^0(\bar Z,
\Omega^1_{\bar Z}(\log(C_{\bar Z}+\bar F+\bar E)))$. And the
$\iota$-invariant subspace is isomorphic to $H^0(\bar Z,
\Omega^1_{\bar Z}(\log(C_{\bar Z}+\bar F))(K_{\bar Z}))$, because
$\iota$-invariant part of the decomposition of
$p_*(\Omega^1_Z(\log(C^1_Z+C^2_Z+F_1+F_2))(K_Z))$ is
$\Omega^1_{\bar Z}(\log(C_{\bar Z}+\bar F))(K_{\bar Z})$.
Therefore, by the Serre duality we have the following proposition.

\begin{proposition}
\label{prop 3.1}We assume that
\begin{enumerate}
\item $\bar S_1,\ldots, \bar S_\ell, \bar G_1, \ldots, \bar G_k, \bar F, \bar E$
are numerically independent in the Picard group of $\bar Z$.
\item The divisor
$\bar S_1+\cdots+\bar S_\ell+\bar G_1+\cdots+\bar G_k+\bar F+\bar
E$ is a simple normal crossing divisor on $\bar Z$.
\item $\bar G_1, \ldots, \bar G_k, \bar F$ are disjoint from two multiple fibers of the
elliptic fibration on $\bar Z$.
\item $2$-sections $\bar S_1,\ldots, \bar S_\ell$ do not meet the exceptional curve $\bar
E$.
\end{enumerate}
Then $H^2(\bar Z, T_{\bar Z}(-\log(\bar S_1+\cdots + \bar
S_\ell+\bar G_1+\cdots+\bar G_k+\bar F)))=0$.
\end{proposition}

By the same argument as in Section 2, we also get the following
proposition.

\begin{proposition}
\label{prop 3.2} With the same assumptions as in Proposition 3.1,
we denote $D_{\bar Z}:=\sum_{i=1}^\ell \bar S_i +\sum_{i=1}^k \bar
G_i+\bar F$. Let $\tau': \bar Z'\to \bar Z$ be a successive
blowing-up of points on $D_{\bar Z}$. Let $D_{\bar Z'}$ be the
reduced divisor of the total transform of $D_{\bar Z}$ or the
reduced divisor of the total transform of  $D_{\bar Z}$ minus some
$(-1)$-curves. Then $H^2(\bar Z', T_{\bar Z'}(-\log D_{\bar
Z'}))=0.$
\end{proposition}

By Proposition 3.2, one can construct surfaces of general type
with $p_g=q=0$ and $\pi_1=\mZ/2\mZ$ by using an Enriques surface
admitting a special elliptic fibration and the methods developed
in \cite{LP1}.

\m According to Kond\=o (Example $\mathrm{II}$ in \cite{Kon}),
there is an Enriques surface $\bar Y$ admitting an elliptic
fibration with a singular fiber of type $\mathrm{I_9}$, a singular
fiber of type $\mathrm{I_1}$, and two 2-sections $\bar S_1$ and
$\bar S_2$. Indeed, we take the 9 curves
 $F_1, F_2, F_3, F_5, F_6, F_7, F_9, F_{10}, F_{11}$
in Fig. 2.4, p. 207 in \cite{Kon}, which form a singular fiber of
type $\mathrm{I_9}$ (on the other hand, the 9 curves
 $F_1, F_2, F_3, F_5, F_8, F_7, F_9, F_{10}, F_{11}$ form a singular fiber of type $2\mathrm{I_9}$). Then by Lemma 3.2,
the elliptic fibration has a singular fiber of type
$\mathrm{I_1}$. Let $F_{\bar Y}$ be a singular fiber of type
$\mathrm{I_1}$. By Lemma 3.1, every 2-section does not pass
through the node of $F_{\bar Y}$. Finally we take the two curves
$F_4$ and $F_8$ as 2-sections $\bar S_1$ and $\bar S_2$. The
configuration of singular fibers and 2-sections on $\bar Y$ is
given in Figure~\ref{figure:Enriques}. Here we rename the components of the
$\mathrm{I_9}$ fibre as $\bar G_1, \bar G_2, \ldots, \bar G_9$.

\begin{figure}
\centering
\includegraphics{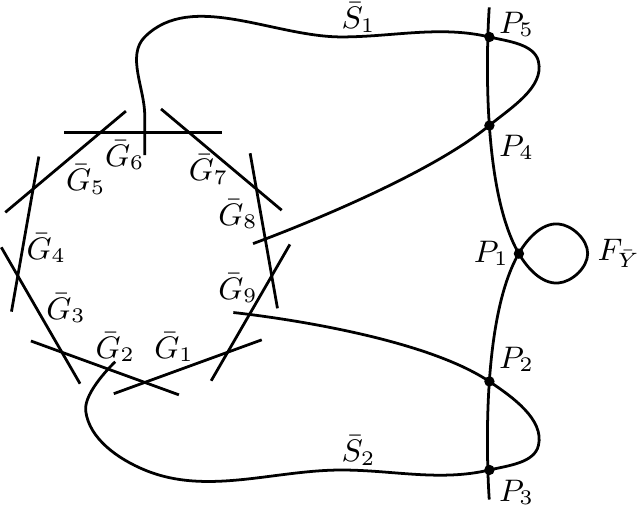}
\caption{Enriques surface}
\label{figure:Enriques}
\end{figure}

\begin{example}
Construction of surfaces of general type with $p_g=0$, $K^2=1$,
and $\pi_1=\mZ/2\mZ$.
\end{example}

Consider the Enriques surface $\bar Y$ in
Figure~\ref{figure:Enriques}. We blow up at $\bar G_6 \cap \bar
S_1$, $\bar G_8 \cap \bar S_1$, $\bar G_2 \cap \bar S_2$, $\bar
G_9 \cap \bar S_2$, and $\bar G_8 \cap \bar G_9$, to obtain a
surface $\tilde Z=\bar Y\# 5\ \overline{\mC\mP^2}$ with four
disjoint linear chains of $\mP^1$'s as shown in
Figure~\ref{figure:K^2=1}.
\[ {\overset{-4}{\circ}}-{\overset{-2}{\circ}}-{\overset{-3}{\circ}}-{\overset{-2}{\circ}},
\,\, \,
{\overset{-4}{\circ}}-{\overset{-2}{\circ}}-{\overset{-3}{\circ}}-{\overset{-2}{\circ}},
\,\, \,
{\overset{-4}{\circ}},
\,\, \,
{\overset{-4}{\circ}}
\]

\m

\begin{figure}
\centering
\includegraphics{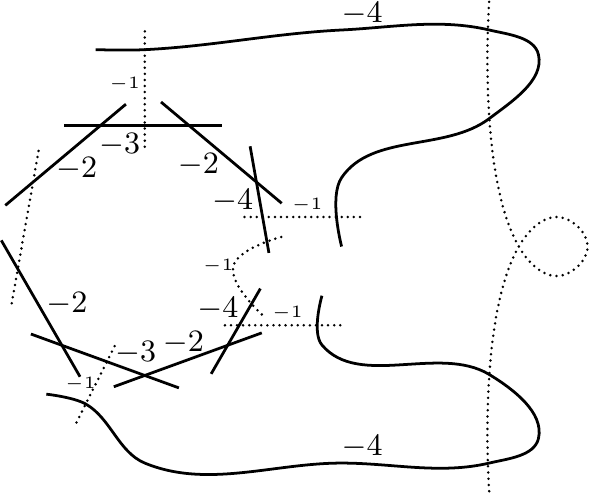}
\caption{$K^2=1$}
\label{figure:K^2=1}
\end{figure}

\m It is not hard to see $\bar S_1, \bar S_2, \bar G_1, \bar G_2,
\bar G_3, \bar G_5, \bar G_6, \bar G_7, \bar G_8, \bar G_9$ are
numerically independent in the Picard group of $\bar Z=\bar Y$ :
Set $a_1\bar G_1+a_2\bar G_2+a_3\bar G_3+a_4\bar G_5+a_5\bar
G_6+a_6\bar G_7+a_7\bar G_8+a_8\bar G_9 +a_9\bar S_1+a_{10}\bar
S_s$. Intersecting with $\bar G_i$, $\bar S_j$, $\bar F$ gives
$a_i=0$.

Let $f: \tilde Z\to X$ be the contraction of the four linear
chains of $\mP^1$'s in $\tilde Z$. By applying
Proposition~\ref{prop 3.1}, Proposition~\ref{prop 3.2}, and
$\mQ$-Gorenstein smoothing theory from \cite{LP1} to the singular
surface $X$, we construct a smooth complex surface $X_t$ of
general type with $p_g=0$ and $K^2=1$. It is easy to check that
the fundamental group of $X_t$ is $\mZ/2\mZ$ by the calculation
based on Van Kampen's theorem (see \cite{LP1}): note that the
index of the singular point obtained by contracting
${\overset{-4}{\circ}}$ is 2, and the index of the singular point
obtained by contracting
${\overset{-4}{\circ}}-{\overset{-2}{\circ}}-{\overset{-3}{\circ}}-{\overset{-2}{\circ}}$
is 3. The two indices are relatively prime.

We claim that the canonical divisor $K_X$, which is $\mQ$-Cartier,
is ample. To see this, we need to check $(f^*K_X).C>0$ for every
irreducible curve $C\subset\tilde Z$, not contracted by $f$. The
adjunction formula gives
$$(f^*K_X).C=K_{\tilde Z}.C+(\sum D_p).C,$$ where $D_p$ is an effective
$\mQ$-divisor supported on $f^{-1}(p)$ for each singular point
$p$.  Since $C$ is not contracted by $f$, $(\sum D_p).C\ge 0$. If
$K_{\tilde Z}.C>0$, then by the adjunction formula,
$(f^*K_X).C>0$. If $K_{\tilde Z}.C<0$, then by Lemma \ref{KC}, $C$
is an exceptional curve for the blowing-up ${\tilde Z}\to
\bar{Y}$, hence a $(-1)$-curve. If $K_{\tilde Z}.C=0$ and
$p_a(C)\ge 1$, then the image $C'$ of $C$ in the Enriques surface
$\bar{Y}$ is 1-dimensional and irreducible. By Lemma \ref{KC},
$C'$ passes through none of $p_i$'s and $p_a(C')=p_a(C)$. If
$p_a(C')\ge 2$, then by the Hodge index theorem, $C'$ intersects
the elliptic configuration $\bar S_1+\bar G_6+\bar G_7+\bar G_8$.
If $p_a(C')=1$, then $C'$ is a fibre or a half fibre of an
elliptic pencil. If $C'$ is linearly equivalent to $\bar S_1+\bar
G_6+\bar G_7+\bar G_8$ or to $2(\bar S_1+\bar G_6+\bar G_7+\bar
G_8)$, then $C'.\bar G_5>0$. If not, $C'.(\bar S_1+\bar G_6+\bar
G_7+\bar G_8)>0$. In any case, $C$ meets at least one of the 4
chains, so $(\sum D_p).C>0$, and hence by the adjunction formula,
$(f^*K_X).C>0$. It remains to check $(f^*K_X).C>0$ for every
$(-1)$-curve $C$ and $(-2)$-curve $C$ not contracted by $f$. It is
easy to check that every $(-2)$-curve $C$ not contracted by $f$
meets at least one of the four chains; by Lemma \ref{KC}, every
$(-2)$-curve $C$ on $\tilde Z$ comes from a $(-2)$-curve $C'$ on
$\bar Y$, and if $C'$ does not intersect the 10 curves $\bar S_1,
\bar S_2, \bar G_1, \bar G_2, \bar G_3, \bar G_5, \bar G_6, \bar
G_7, \bar G_8, \bar G_9$, then the 11 curves will be numerically
independent, a contradiction. So $(\sum D_p).C>0$ and hence,
$(f^*K_X).C>0$. For each of the five $(-1)$-curves, a direct
computation of $D_p$ shows that $(\sum D_p).C>1$, hence
$(f^*K_X).C>0$.

Finally, note that the ampleness is an open condition for a proper
morphism (cf. Proposition 1.41 in \cite{KM}), so a general fiber
of a $\mQ$-Gorenstein smoothing of $X$ has ample canonical class.

\begin{example} Construction of surfaces of general type with $p_g=0$, $K^2=2$,
and $\pi_1=\mZ/2\mZ$. \end{example}

Consider the Enriques surface $\bar Y$ again in Figure~\ref{figure:Enriques}. We blow
up at $P_1$, $P_4$, and three times at $P_5$, and $\bar G_6 \cap \bar G_7$, and $\bar G_8 \cap \bar S_1$ . Then we get
a surface $\tilde Z=\bar Y\# 7\ \overline{\mC\mP^2}$ with three
disjoint linear chains of $\mP^1$'s as shown in Figure~\ref{figure:K^2=2}.
\[ {\overset{-6}{\circ}}-{\overset{-2}{\circ}}-{\overset{-2}{\circ}}
, \, \, \,
{\overset{-7}{\circ}}-{\overset{-3}{\circ}}-{\overset{-2}{\circ}}
-{\overset{-2}{\circ}}-{\overset{-2}{\circ}}-{\overset{-2}{\circ}}, \, \, \,
{\overset{-3}{\circ}}-{\overset{-3}{\circ}}
\]

\m
\begin{figure}
\centering
\includegraphics{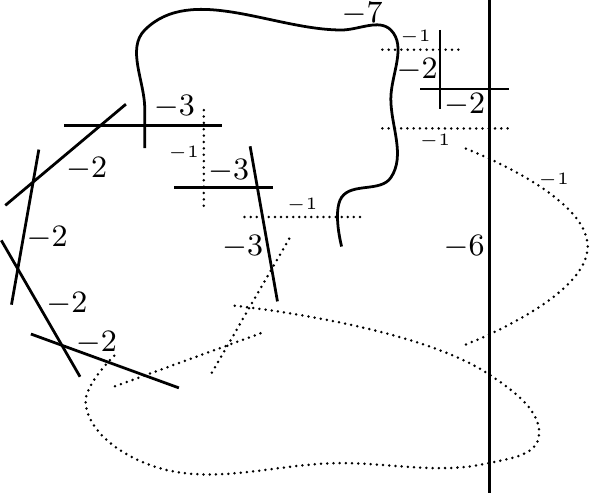}
\caption{$K^2=2$}
\label{figure:K^2=2}
\end{figure}

\m It is not hard to see $\bar S_1, \bar G_2, \bar G_3, \bar G_4,
\bar G_5, \bar G_6, \bar G_7, \bar G_8,\bar F, \bar E$ are
numerically independent in the Picard group of $\bar Z$, the
blow-up of $\bar Y$ at the node $P_1$ of the nodal fibre: Set
$a_1\bar S_1+a_2\bar G_2+a_3\bar G_3+a_4\bar G_4+a_5\bar
G_5+a_6\bar G_6+a_7\bar G_7+a_8\bar G_8+a_9\bar F+a_{10}\bar E=0$.
Intersecting with $\bar G_i$, $\bar S_j$, $\bar F$ gives $a_i=0$.
By applying $\mQ$-Gorenstein smoothing theory to the singular
surface $X$ obtained by contracting three linear chains of
$\mP^1$'s in $\tilde Z$, we construct a complex surface of general
type with $p_g=0$ and $K^2=2$. Similarly, one can check that the
fundamental group of this surface is $\mZ/2\mZ$: note that the
curve $\bar G_9$ in Figure~\ref{figure:K^2=2} meets only one end
curve in a linear chain of $\mP^1$ which is contracted.

By the same argument as in the case of $K^2=1$, a general fiber of
a $\mQ$-Gorenstein smoothing of $X$ has ample canonical class. In
this case, it is simpler to check the ampleness of $K_X$, as we
contract the proper transform of a fibre.

\begin{example} Construction of surfaces of general type with $p_g=0$,
$K^2=3$, and $\pi_1=\mZ/2\mZ$. \end{example}

Again consider the Enriques surface $\bar Y$ in Figure~\ref{figure:Enriques}. We blow
up at $P_1$, $P_2$, $P_3$, six times at $P_5$, twice at $P_4$, and
at one of the two intersection points between $S_2$ and the
singular fiber of type $\mathrm{I_9}$, to get a surface $\tilde
Z=\bar Y\# 12\ \overline{\mC\mP^2}$ with three disjoint linear
chains of $\mP^1$'s as shown in Figure~\ref{figure:K^2=3}.
\[ {\overset{-5}{\circ}}-{\overset{-2}{\circ}}, \, \, \,
{\overset{-9}{\circ}}-{\overset{-2}{\circ}}-{\overset{-2}{\circ}}
-{\overset{-2}{\circ}}-{\overset{-2}{\circ}}-{\overset{-2}{\circ}},
\, \, \,
{\overset{-2}{\circ}}-{\overset{-9}{\circ}}-{\overset{-2}{\circ}}-{\overset{-2}{\circ}}
-{\overset{-2}{\circ}}-{\overset{-2}{\circ}}-{\overset{-3}{\circ}}
\]

\m
\begin{figure}
\centering
\includegraphics{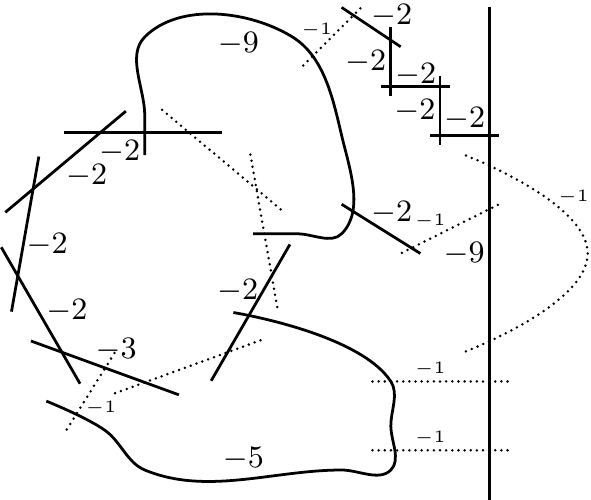}
\caption{$K^2=3$}
\label{figure:K^2=3}
\end{figure}

\m Similarly, we see that $\bar S_1, \bar S_2, \bar G_2, \bar G_3,
\bar G_4, \bar G_5, \bar G_6, \bar G_9, \bar F, \bar E$ are
numerically independent in the Picard group of $\bar Z$, the
blow-up of $\bar Y$ at the node $P_1$ of the nodal fibre: Set
$a_1\bar S_1+a_2\bar S_2+a_3\bar G_2+a_4\bar G_3+a_5\bar
G_4+a_6\bar G_5+a_7\bar G_6+a_8\bar G_9+a_9\bar F+a_{10}\bar E=0$.
Intersecting with $\bar G_7$ gives $a_7=0$, then intersecting with
$\bar G_8$ gives $a_1+a_8=0$. Intersecting with $\bar G_1$ gives
$a_3+a_8=0$, and intersecting with $\bar G_6$ gives $a_1+a_6=0$ by
using $a_7=0$. We have $-2a_8+a_2=0$ by intersecting with $\bar
G_9$. Then intersecting with $\bar G_2$ produces $a_2-2a_3+a_4=0$,
and intersecting with $\bar G_3$ gives $a_3-2a_4+a_5=0$.
Intersecting with $\bar G_4$ gives $a_4-2a_5+a_6=0$. These
relations give $a_8=a_6=-a_1, a_3=a_1, a_2=-2a_1, a_4=4a_1,
a_5=7a_1,$ and $a_6=10a_1$, But intersecting with $\bar G_5$ gives
$-2a_6+a_5=0$. So $a_1=0$, and $a_2=a_3=a_4=a_5=a_6=a_8=0$.
Finally intersecting with $\bar S_1$ gives $a_9=0$, and
intersecting with $\bar F$ produces $a_{10}=0$.

By applying $\mQ$-Gorenstein smoothing theory to the singular
surface $X$ obtained by contracting three linear chains of
$\mP^1$'s in $\tilde Z$, we construct a complex surface of general
type with $p_g=0$ and $K^2=3$. It is easy to check that the
fundamental group of this surface is $\mZ/2\mZ$: note that the
index of the singular point obtained by contracting
${\overset{-5}{\circ}}-{\overset{-2}{\circ}}$ is 3, and the index
of the singular point obtained by contracting
${\overset{-9}{\circ}}-{\overset{-2}{\circ}}-{\overset{-2}{\circ}}
-{\overset{-2}{\circ}}-{\overset{-2}{\circ}}-{\overset{-2}{\circ}}$
is 7. The two indices are relatively prime.

As in the previous cases, one can show that a general fiber of a
$\mQ$-Gorenstein smoothing of $X$ has ample canonical class.

\begin{remark} One can use other Enriques surfaces. For example,
take the Enriques surface, Example $\mathrm{VII}$ in \cite{Kon}.
On this Enriques surface $\bar Y$, there is an elliptic fibration
with a singular fiber of type $\mathrm{I_9}$, a singular fiber of
type $\mathrm{I_1}$, and two 2-sections. Indeed, we take the 9
curves
 $E_1, E_2, E_3, E_4, E_5, E_6, E_7, E_8, E_9$
in Fig. 7.7, p. 233 in \cite{Kon}. These form a singular fiber of
type $\mathrm{I_9}$, as it splits in the K3 cover of $\bar Y$ as
we see in Fig.7.3, p. 230. Then by Lemma 3.2, the elliptic
fibration has a singular fiber of type $\mathrm{I_1}$. Let
$F_{\bar Y}$ be a singular fiber of type $\mathrm{I_1}$. By Lemma
3.1, every 2-section passes through two smooth points of $F_{\bar
Y}$. Finally we take the two curves $E_{10}$ and $E_{11}$ as
2-sections. Let $E_{10}\cap E_{11}=\{P_1\}$, $E_{10}\cap F_{\bar
Y}=\{P_2, P_3\}$, $E_{11}\cap F_{\bar Y}=\{P_4, P_5\}$. Blowing up
once at $P_1, P_2, P_3, P_4$, the node of $F_{\bar Y}$, and 5
times at $P_5$, we get a surface $\tilde Z=\bar Y\# 10\
\overline{\mC\mP^2}$ with three disjoint linear chains of
$\mP^1$'s
\[ {\overset{-5}{\circ}}-{\overset{-2}{\circ}}, \, \, \,
{\overset{-9}{\circ}}-{\overset{-2}{\circ}}-{\overset{-2}{\circ}}
-{\overset{-2}{\circ}}-{\overset{-2}{\circ}}-{\overset{-2}{\circ}},
\, \, \,
{\overset{-8}{\circ}}-{\overset{-2}{\circ}}-{\overset{-2}{\circ}}-{\overset{-2}{\circ}}
-{\overset{-2}{\circ}},\] which leads to a construction of
surfaces of general type with $p_g=0$, $K^2=3$, and
$\pi_1=\mZ/2\mZ$.
\end{remark}

\begin{example} Construction of surfaces of general type with $p_g=0$,
$K^2=4$, and $\pi_1=\mZ/2\mZ$. \end{example}

We blow up at $P_1$, $P_2$, $P_3$, $S_1 \cap \bar{G}_6$, and three times at $P_5$, and eight times at
$\bar{G}_6 \cap \bar{G}_7$ on the Enriques surface $\bar Y$ in Figure~\ref{figure:Enriques}. We then get a surface $\tilde
Z=\bar Y\# 15\ \overline{\mC\mP^2}$ with two disjoint linear
chains of $\mP^1$'s as shown in Figure~\ref{figure:K^2=4}.
\begin{align*}
&\overset{-2}{\circ}-\overset{-2}{\circ}-\overset{-9}{\circ}-\overset{-2}{\circ}-\overset{-2}
{\circ}-\overset{-2}{\circ}-\overset{-2}{\circ}-\overset{-4}{\circ},\\
&\overset{-2}{\circ}-\overset{-2}{\circ}-\overset{-7}{\circ}-\overset{-6}{\circ}-\overset{-2}
{\circ}-\overset{-3}{\circ}-\overset{-2}{\circ}-\overset{-2}{\circ}-\overset{-2}{\circ}-\overset{-2}{\circ}-\overset{-4}{\circ}
\end{align*}
Similarly, we see that $\bar S_1, \bar S_2, \bar G_2, \bar G_3,
\bar G_4, \bar G_5, \bar G_6, \bar G_7, \bar G_8, \bar F, \bar E$
are numerically independent in the Picard group of $\bar Z$, the
blow-up of $\bar Y$ at the node $P_1$ of the nodal fibre. \m
\begin{figure}
\centering
\includegraphics{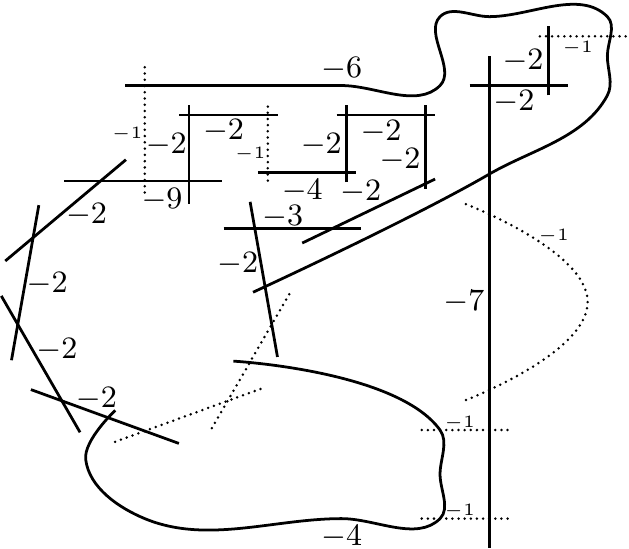}
\caption{$K^2=4$}
\label{figure:K^2=4}
\end{figure}

By applying $\mQ$-Gorenstein smoothing theory  as in \cite{LP1} to
the singular surface $X$ obtained by contracting two linear chains
of $\mP^1$'s in $\tilde Z$, we construct a complex surface of
general type with $p_g=0$ and $K^2=4$. It is easy to check that
the fundamental group of this surface is $\mZ/2\mZ$ by the same
method. And as in the previous cases, a general fiber of a
$\mQ$-Gorenstein smoothing of $X$ has ample canonical class.

\begin{example} Construction of a symplectic $4$-manifold with $b_2^+=1$, $K^2=5$, and $\pi_1 =
\mathbb{Z}/2\mathbb{Z}$. \end{example}

We consider the Enriques surface in Figure~\ref{figure:Enriques}.
According to Kond\=o~\cite{Kon} the Enriques surface has two
$I_1$-singular fibers as in Figure~\ref{figure:K^2=5-Y}. We blow up
five times totally at the five marked points $\bullet$ as in
Figure~\ref{figure:K^2=5-Y}. We blow up again three times and four
times at the two marked points $\bigodot$, respectively. We then get
a surface  $Z=\bar Y\# 12\ \overline{\mC\mP^2}$;
Figure~\ref{figure:K^2=5-Z}. There exist two disjoint linear chains
of $\mathbb{CP}^1$'s in $Z$:
\begin{align*}
&\overset{-6}{\circ}-\overset{-2}{\circ}-\overset{-2}{\circ} \\
&\overset{-5}{\circ}-\overset{-8}{\circ}-\overset{-6}{\circ}-\overset{-2}
{\circ}-\overset{-3}{\circ}-\overset{-2}{\circ}-\overset{-2}{\circ}-\overset{-2}{\circ}
-\overset{-2}{\circ}-\overset{-2}{\circ}-\overset{-3}{\circ}-\overset{-2}{\circ}-\overset{-2}{\circ}-\overset{-2}{\circ}
\end{align*}

\begin{figure}[hbtb]
\centering
\includegraphics{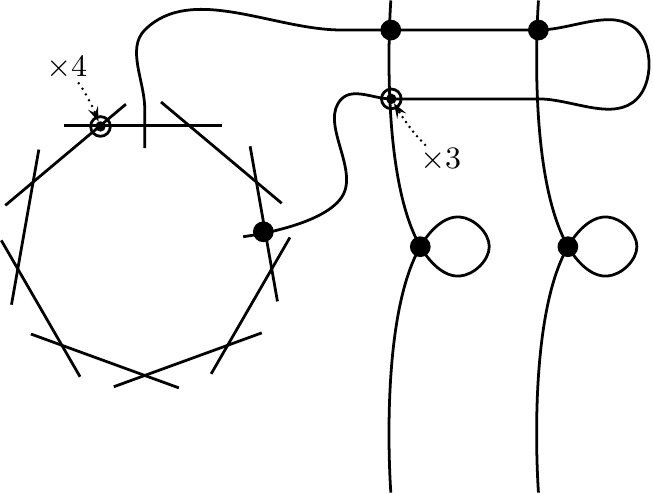}
\caption{The Enriques surface}
\label{figure:K^2=5-Y}
\end{figure}

\begin{figure}[hbtb]
\centering
\includegraphics{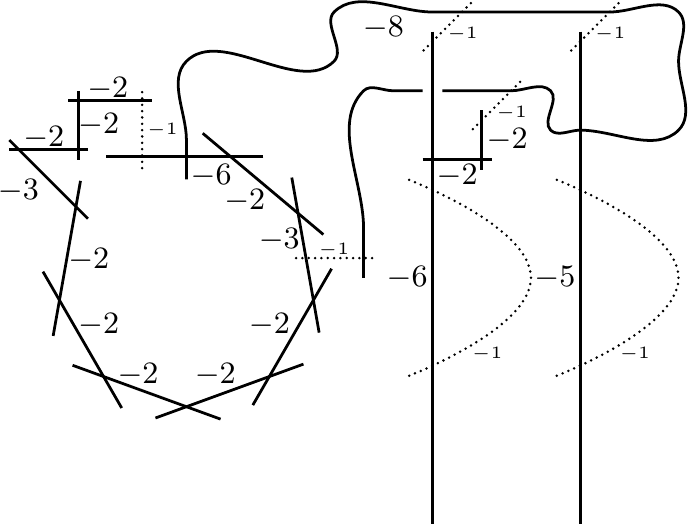}
\caption{Symplectic $K^2=5$}
\label{figure:K^2=5-Z}
\end{figure}

We now perform a rational blow-down surgery of the surface $Z=\bar
Y\# 12\ \overline{\mC\mP^2}$. The rational blow-down $\tilde{Z}$
is a symplectic $4$-manifold. Thus we get a symplectic
$4$-manifold $\tilde{Z}$ with $b_2^+=1$ and $K^2=5$. It is easy to
show that $\pi_1(Z_{151,4})=\mathbb{Z}/2\mathbb{Z}$.

\begin{enumerate}
\item One can prove that the symplectic $4$-manifold $\widetilde{Z}$ constructed above is
minimal by using a technique in Ozsv\'{a}th and Szab\'{o}~\cite{Ozsvath-Szabo}.

\item It is an intriguing question whether the symplectic $4$-manifold $\widetilde{Z}$ admit a complex structure.
Since the cohomology $H^2 (T^0_{X})$ is not zero in this case, it is
hard to determine whether there exists a global
$\mathbb{Q}$-Gorenstein smoothing. We leave this question for future
research.
\end{enumerate}

\begin{remark}
 A surface $X$ of general type with $p_g=0$, $K^2=k
 (1\le k\le 7)$, and $\pi_1=\mZ/2\mZ$ provides an exotic structure
on $3\mC\mP^2\# (19-2k)\overline{\mC\mP^2}$. The universal double
cover $Y$ of $X$ is a simply connected surface of general type
with $p_g=1$, $c_2=24-2k$, $b_2^+=3$, $b_2^-=19-2k$. Its index
$\sigma=16-2k$ is not divisible by 16, so by Rohlin's Theorem
\cite{Roh} the intersection form on $H^2(Y,\mZ)$ is odd and then
by Freedman's Theorem \cite{Free} Y is homeomorphic to
$3\mC\mP^2\# (19-2k)\overline{\mC\mP^2}$. By a result of Donaldson
\cite{Don} or by a result of Friedman and Qin \cite{FQ}, $Y$ is
not diffeomorphic to $3\mC\mP^2\# (19-2k)\overline{\mC\mP^2}$.
\end{remark}

\section{The case of an  $E(3)$ surface}
\label{sec-4}

In this section, we give a sufficient condition for the existence
of a $\mQ$-Gorenstein smoothing of a singular surface obtained
from an  $E(3)$ surface (Proposition \ref{pro-4.1}). We also show
that if the singular surface is obtained by contracting two
disjoint sections and other curves, then it always has non-trivial
obstruction space (Proposition \ref{pro-4.2}).

 Let $Y$ be an $E(3)$ surface. Let $F$
be a general fiber of the elliptic fibration $f: Y\to \mP^1$,
which is a smooth elliptic curve. Let $C$ be a section (it is a
$(-3)$-curve), and let  $G_1,\ldots, G_k$ be $(-2)$-curves in the
union of singular fibers. Assume that the support of $\cup_{i=1}^k
G_i$ does not contain the support of a whole singular fiber. We
note that $K_Y=F$. Set $G:=G_1+\cdots +G_k$, and
$\cO_G:=\cO_{G_1}\oplus\cdots\oplus\cO_{G_k}$ for abbreviation.

\begin{proposition}
\label{pro-4.1} With the assumptions and the notation as above,
assume further that $G_1, \ldots, G_k$, $F, C$
are numerically independent in the Picard group of $Y$. Then $H^0(Y, \Omega_Y^1(\log(C+G))(F))=0$.
\end{proposition}

\begin{proof}
The proof is also obtained by the following commutative diagram and
the snake lemma.
\[\begin{array}{ccccccc}
& 0 & & 0& &  & \\
& \downarrow & & \downarrow & &  & \\
0\to & \Omega_Y^1 & \to & \Omega_Y^1(F) & \to & \Omega_Y^1(F)\otimes\cO_F &\to 0 \\
& \downarrow & & \downarrow & & \downarrow & \\
0\to & \Omega_Y^1(\log(C+G+F)) & \to & \Omega_Y^1(\log(C+G))(F) & \to & \cK &\to 0 \\
& \downarrow & & \downarrow & & \downarrow & \\
0\to\cO_F\to & \cO_C\oplus\cO_G\oplus\cO_F & \to &
\cO_C(F)\oplus\cO_G(F) & \to & \mC_p &\to
0\\
& \downarrow & & \downarrow & & \downarrow & \\
& 0 & & 0& & 0 &
\end{array}\]
where $\cK$ is the cokernel of the map from
$\Omega_Y^1(\log(C+G+F))$ to $\Omega_Y^1(\log(C+G))(F)$ and $p$ is
the intersection point of $F$ and $C$.

We have a short exact sequence
\[0\to \cO_F\to \Omega_Y^1(F)\otimes\cO_F \to \Omega_F^1(F)\to 0,\]
and then by the snake lemma, another short exact sequence
\[0\to \Omega_F^1(F)\to \cK\to \mC_p\to 0.\]
We get $\cK=\omega_F(p)$. So, $h^0(\cK)=1$. By the same argument
as in the proof of Lemma 2 of \cite{LP1}, we get $H^0(Y,
\Omega_Y^1(F))=0$. Therefore $H^0(F, \omega_F(F))\cong H^0(F,\cK)$
maps injectively into $H^1(Y, \Omega_Y^1)$, and its image contains
no non-zero vector of the image of
$H^0(\cO_C\oplus\cO_G\oplus\cO_F)$. It implies that $H^0(F, \cK)$
maps injectively into $H^1(Y, \Omega^1_Y(\log(C+G+F)))$.

Since $H^0(Y, \Omega_Y^1)=0$ and the first Chern class map from
$H^0(\cO_C\oplus\cO_G\oplus\cO_F)$ to $H^1(Y, \Omega_Y^1)$ is
injective by the assumption, we get the vanishing $H^0(Y,
\Omega_Y^1(\log(C+G+F)))=0$. And then we have the vanishing $H^0(Y,
\Omega_Y^1(\log(C+G))(F))=0$.
\end{proof}

By the Serre duality, $H^2(Y, T_Y(-\log(C+G)))=0$. But if we choose
two disjoint sections $C_1$ and $C_2$, then $H^0(Y,
T_Y(-\log(C_1+C_2+G)))\ne 0$.

\begin{proposition}
\label{pro-4.2} $H^0(Y, \Omega_Y^1(\log(C_1+C_2))(F))\ne 0$.
\end{proposition}

\begin{proof}
Consider the following commutative diagram as before.
\[\begin{array}{ccccccc}
& 0 & & 0& &  & \\
& \downarrow & & \downarrow & &  & \\
0\to & \Omega_Y^1 & \to & \Omega_Y^1(F) & \to & \Omega_Y^1(F)\otimes\cO_F &\to 0 \\
& \downarrow & & \downarrow & & \downarrow & \\
0\to & \Omega_Y^1(\log(C_1+C_2+F)) & \to & \Omega_Y^1(\log(C_1+C_2))(F) & \to & \cK &\to 0 \\
& \downarrow & & \downarrow & & \downarrow & \\
0\to\cO_F\to & \cO_{C_1}\oplus\cO_{C_2}\oplus\cO_F & \to &
\cO_{C_1}(F)\oplus\cO_{C_2}(F) & \to & \mC_p\oplus\mC_q &\to
0\\
& \downarrow & & \downarrow & & \downarrow & \\
& 0 & & 0& & 0 &
\end{array}\]
where $\cK$ is the cokernel of the map from
$\Omega_Y^1(\log(C_1+C_2+F))$ to $\Omega_Y^1(\log(C_1+C_2))(F)$
and $p=F\cap C_1$, $q=F\cap C_2$.

We have a short exact sequence
\[0\to \cO_F\to \Omega_Y^1(F)\otimes\cO_F \to \Omega_F^1(F)\to 0.\]
Then by the snake lemma, we have another short exact sequence
\[0\to \Omega_F(F)\to \cK\to \mC_p\oplus\mC_q\to 0,\]
and we get $\cK=\omega_F(p+q)$. Then the 1-dimensional subspace of
$H^0(F, \cK)$, induced by the kernel of the map from
$H^0(\mC_{p,q})$ to $H^1(F, \Omega_F^1(F))$, maps to 0 in $H^1(Y,
\Omega_Y^1(\log(C_1+C_2+F)))$. Therefore,
$H^0(\Omega_Y^1(\log(C_1+C_2))(F))$ is a 1-dimensional space.
\end{proof}

By Proposition \ref{pro-4.2},  we cannot obtain from an $E(3)$
surface a singular surface $X$ with $K_X$ big, if we impose the
vanishing of the obstruction space of $X$. Thus, to construct a
surface of general type with $p_g=2$ and $q=0$, one cannot impose
the vanishing of the obstruction space, and need to find a
singular surface with an automorphism such that the obstruction
space has trivial invariant part.

\begin{example} Construction of a surface of general type
with $p_g=2$ and $q=0$.
\end{example}  Let
$D_1, D_2$ be two smooth conics in $\mP^2$ such that $D_1$ and
$D_2$ meet transversally at four points $p_1, \ldots, p_4$. Let
$T$ be a smooth plane curve of degree 4 meeting $D_1, D_2$
transversally at four points $p_1, \ldots, p_4$. Let $V$ be a
$(\mZ/2\mZ)^2$-cover of $\mP^2$ branched over $D_1, D_2$, and $T$.
Then $V$ has four $\dfrac{1}{4}(1,1)$ singularities over $p_1,
\ldots, p_4$, and $p_g(V)=2$, $\chi(\cO_V)=3$, $K_V^2=4$. Its
minimal resolution is an $E(3)$ surface with an elliptic fibration
induced by the double cover of the pencil of conics $D_t$
generated by $D_1$ and $D_2$ branched over the four intersection
points of $D_t$ and $T$ away from the four points $p_1, \ldots,
p_4$. Then by the same argument as in \cite{Lee} one can construct
minimal surfaces of general type with $p_g=2$, $q=0$, and $1\le
K^2\le 4$.

\m

{\em Acknowledgements}. Yongnam Lee would like to thank Noboru
Nakayama for valuable discussions during the workshop in June,
2009 at RIMS in Kyoto University. Yongnam Lee was supported by the
WCU Grant funded by the Korean Government (R33-2008-000-10101-0),
and partially supported by the Special Research Grant of Sogang
University. JongHae Keum was supported by the National Research
Foundation of Korea funded by the Ministry of EST
(NRF-2007-C00002).

\end{document}